\documentclass[10pt,a4paper]{amsart}
\usepackage[applemac]{inputenc}
\usepackage{latexsym}
\usepackage{color,graphicx,shortvrb}
\usepackage{amsmath, amssymb}
\usepackage{amsfonts}
\usepackage[colorlinks, bookmarks=true]{hyperref}

\newtheorem{theorem}{Theorem}[section]
\newtheorem{lemma}[theorem]{Lemma}

\newtheorem{corollary}[theorem]{Corollary}

\theoremstyle{definition}

\newtheorem{remark}[theorem]{Remark}

\author{A. G. Aksoy, J. M. Almira}
\title{ On Montel and Montel-Popoviciu theorems in several variables}

\begin{document}

\begin{abstract}
We present an elementary proof of a general version of Montel's theorem in several variables which is based on the use of tensor product polynomial interpolation. We also prove a Montel-Popoviciu's type theorem for functions $f:\mathbb{R}^d\to\mathbb{R}$ for $d>1$. Furthermore, our proof of this result is also valid for the case $d=1$, differing in several points from Popoviciu's original proof. Finally, we demonstrate that our results are optimal.
\end{abstract}







\maketitle


\section{Introduction}
The study of functional equations has substantially grown in the last three decades \cite{A_D}, \cite{B_J}, \cite{jarai}, \cite{J_Laszlo}, \cite{pales}, \cite{R_B}. Research in this area is provoking interesting questions concerning characterizations of polynomials \cite{AK_CJM}, \cite{L} and exponential polynomials \cite{laszlo1}, \cite{laszlo_mona}. These concrete questions have close connections to spectral analysis and synthesis and  have found its way to a number of interesting applications \cite{laszlo1}, \cite{laszlo_discrete}, \cite{laszlo_harm}. Both Montel and Popoviciu, in their seminal papers \cite{montel_1935}, \cite{montel} and \cite{popoviciu}, used Fr\'{e}chet functional equation, with some additional regularity conditions, for the characterization of polynomials. In this paper, we are interested in generalizing these results, by using some new tools, in the several variable setting. Our aim is  to show that under suitable conditions we again end up with polynomials.

Concretely, we are interested in a special regularity result for the the functional equation $\Delta_h^{m+1}f(x)=0$, where $f:\mathbb{R}^d\to\mathbb{R}$ and the higher differences operator $\Delta^{m+1}_h$  is inductively defined by $\Delta_h^1f(x)=f(x+h)-f(x)$, and $\Delta_h^{n+1}f(x)=\Delta_h(\Delta_h^nf)(x)$, $n=1,2,\cdots$. This equation was introduced in the literature for functions $f:\mathbb{R}\to\mathbb{R}$,  by M. Fr\'{e}chet in 1909 as a particular case of the functional equation
\begin{equation}\label{fre}
\Delta_{h_1h_2\cdots h_{m+1}}f(x)=0 \ \ (x,h_1,h_2,\dots,h_{m+1}\in \mathbb{R}),
\end{equation}
where $f:\mathbb{R}\to\mathbb{R}$ and $\Delta_{h_1h_2\cdots h_s}f(x)=\Delta_{h_1}\left(\Delta_{h_2\cdots h_s}f\right)(x)$, $s=2,3,\cdots$. In particular, after Fr\'{e}chet's
seminal paper \cite{frechet}, the solutions of \eqref{fre} are named ``polynomial functions'' by the functional equations community, since it is known that, under very mild regularity conditions on $f$, if $f:\mathbb{R}\to\mathbb{R}$ satisfies \eqref{fre}, then $f(x)=a_0+a_1x+\cdots a_{m}x^{m}$ for all $x\in\mathbb{R}$ and certain constants $a_i\in\mathbb{R}$. For example, in order to have this property, it is enough for $f$ being  bounded on a set $A\subseteq \mathbb{R}$ of positive Lebesgue measure $|A|>0$ (see, for example, \cite{laszlo1} for a proof of this result). Equation \eqref{fre} can  also be studied for functions $f:X\to Y$  whenever $X, Y$ are two  $\mathbb{Q}$-vector spaces and the variables $x,h_1,\cdots,h_{m+1}$ are assumed to be elements of $X$:
\begin{equation}\label{fregeneral}
\Delta_{h_1h_2\cdots h_{m+1}}f(x)=0 \ \ (x,h_1,h_2,\dots,h_{m+1}\in X).
\end{equation}
In this context, the general solutions of \eqref{fregeneral} are characterized as functions of the form $f(x)=A_0+A_1(x)+\cdots+A_m(x)$, where $A_0$ is a constant and $A_k(x)=A^k(x,x,\cdots,x)$ for a certain $k$-additive symmetric function $A^k:X^k\to Y$ (we say that $A_k$ is the diagonalization of $A^k$). In particular, if $x\in X$ and $r\in\mathbb{Q}$, then $f(rx)=A_0+rA_1(x)+\cdots+r^mA_m(x)$. Furthermore, it is known that $f:X\to Y$ satisfies \eqref{fregeneral} if and only if it satisfies
\begin{equation}\label{frepasofijo}
\Delta_{h}^{m+1}f(x):=\sum_{k=0}^{m+1}\binom{m+1}{k}(-1)^{m+1-k}f(x+kh)=0 \ \ (x,h\in X).
\end{equation}
A proof of this fact follows directly from Djokovi\'{c}'s Theorem \cite{Dj} (see also \cite[Theorem 7.5, page 160]{HIR}, \cite[Theorem 15.1.2., page 418]{kuczma} and, for a completely different new proof, \cite{laszlo_mona}).

In 1935 P. Montel  \cite{montel_1935} studied Fr\'{e}chet's functional equation from a fresh perspective (see also \cite{montel}). Indeed, he was not motivated by Fr\'{e}chet's paper but by a much older one by Jacobi \cite{Jacobi}, who  in 1834 proved that  if $f:\mathbb{C}\to\widehat{\mathbb{C}}$ is a non-constant meromorphic function defined on the complex numbers, then  $\mathfrak{P}_0(f)=\{w\in\mathbb{C}:f(z+w)=f(z) \text{ for all } z\in\mathbb{C}\}$, the set of periods of $f$,  is a discrete subgroup of $(\mathbb{C},+)$. This reduces the possibilities to the following three cases: $\mathfrak{P}_0(f)=\{0\}$, or $\mathfrak{P}_0(f)=\{nw_1:n\in\mathbb{Z}\}$ for a certain complex number $w_1\neq 0$, or $\mathfrak{P}_0(f)=\{n_1w_1+n_2w_2:(n_1,n_2)\in \mathbb{Z}^2\}$ for certain complex numbers $w_1,w_2$ satisfying $w_1,w_2\neq 0$ and $w_1/w_2\not\in\mathbb{R}$. In particular, these functions cannot have three independent periods and there exist meromorphic functions  $f:\mathbb{C}\to\widehat{\mathbb{C}}$  with two independent periods $w_1,w_2$ as soon as $w_1/w_2\not\in\mathbb{R}$. These functions are called doubly periodic (or elliptic) and have an important role in complex function theory \cite{JS}.  Analogously, if the function $f:\mathbb{R}\to\mathbb{R}$ is continuous and non-constant, it does not admit  two
$\mathbb{Q}$-linearly independent periods.  Obviously,  Jacobi's theorem can be formulated as a result which characterizes the constant functions as those meromorphic  functions $f:\mathbb{C}\to\widehat{\mathbb{C}}$  which solve a system of functional equations of the form
\begin{equation}\label{JC}
\Delta_{h_1}f(z)=\Delta_{h_2}f(z)=\Delta_{h_3}f(z)=0 \ \ (z\in \mathbb{C}),
\end{equation}
for three independent periods $\{h_1,h_2,h_3\}$ (i.e., $h_1\mathbb{Z}+h_2\mathbb{Z}+h_3\mathbb{Z}$ is a dense subset of  $\mathbb{C}$). For the real case, the result states that, if $h_1,h_2\in\mathbb{R}\setminus\{0\}$ are two nonzero real numbers and $h_1/h_2\not\in\mathbb{Q}$,  the continuous function $f:\mathbb{R}\to\mathbb{R}$ is a constant function if and only if it solves the system of functional equations
\begin{equation} \label{JR}
\Delta_{h_1}f(x)=\Delta_{h_2}f(x)= 0\ (x\in \mathbb{R}).
\end{equation}
In \cite{montel_1935}, \cite{montel} Montel  substituted  $\Delta^{m+1}_h$ for $\Delta_h$ in the equations $(\ref{JC}),(\ref{JR})$ and proved that these equations are appropriate for the characterization of ordinary polynomials.
\begin{theorem}[Montel] \label{monteldimuno} Assume that $f:\mathbb{C}\to\mathbb{C}$ is an analytic function  which solves a system of functional equations of the form
\begin{equation}\label{JCM}
\Delta_{h_1}^{m+1}f(z)=\Delta_{h_2}^{m+1}f(z)=\Delta_{h_3}^{m+1}f(z)=0 \ \ (z\in \mathbb{C})
\end{equation}
for three independent periods $\{h_1,h_2,h_3\}$. Then $f(z)=a_0+a_1z+\cdots+a_mz^m$ is an ordinary polynomial with complex coefficients and degree $\leq m$. Furthermore, if $\{h_1,h_2\}\subset \mathbb{R}\setminus \{0\}$ satisfy $h_1/h_2\not\in\mathbb{Q}$,
 the continuous function $f:\mathbb{R}\to\mathbb{R}$ is an ordinary polynomial with real coefficients and degree $\leq m$ if and only if it solves the system of functional equations
\begin{equation} \label{JRM}
\Delta_{h_1}^{m+1}f(x)=\Delta_{h_2}^{m+1}f(x)= 0\ (x\in \mathbb{R}).
\end{equation}
\end{theorem}
To differentiate the relationship between his theorem and Jacobi's results, Montel named ``generalized periods'' of order $m+1$ of $f$ to the vectors $h$ such that
$\Delta_{h}^{m+1}f= 0$.

Montel's result uses the  regularity properties of Fr\'{e}chet's functional equation in a new non-standard form. Indeed, the idea is now not to conclude the regularity of $f$ from the assumption that it solves the equation for all  $h$ and it satisfies some mild regularity condition, but to assume that $f$ is globally regular (indeed, it is continuous everywhere) and to describe the minimal set  $\Gamma$ of generalized periods $h$ such that, if  the equation is solved for all $h\in \Gamma$, then it is also solved for all $h$. Thus, the regularity of the solution is assumed and used to conclude that, in order to be a solution of Fr\'{e}chet's equation it is enough to have a very small set of generalized periods.

In his paper \cite{montel}, Montel also studied the equation $\eqref{frepasofijo}$ for $X=\mathbb{R}^d$, with $d>1$,  and $f:\mathbb{R}^d\to\mathbb{C}$ continuous, and for $X=\mathbb{C}^d$ and $f:\mathbb{C}^d\to \mathbb{C}$ analytic. Concretely, he stated (and gave a proof for $d=2$) the following result.
\begin{theorem}[Montel's Theorem in several variables] \label{montelvvcf}  Let $\{h_1,\cdots,h_{s}\}\subset \mathbb{R}^d$  be such that
\begin{equation} \label{perm}
h_1\mathbb{Z}+h_2\mathbb{Z}+\cdots+h_{s}\mathbb{Z} \text{ is a dense subset of } \mathbb{R}^d,
\end{equation}
and let $f\in C(\mathbb{R}^d,\mathbb{C})$ be such that $\Delta_{h_k}^m(f) =0$, $k=1,\cdots,s$. Then  $f(x)=\sum_{|\alpha|<N}a_{\alpha}x^{\alpha}$ for some $N\in\mathbb{N}$, some complex numbers $a_{\alpha}$, and all $x\in\mathbb{R}^d$. Thus, $f$ is an ordinary complex-valued polynomial in $d$ real variables.

Consequently, if $d=2k$, $\{h_i\}_{i=1}^{s}$ satisfies $\eqref{perm}$, the function  $f:\mathbb{C}^k\to \mathbb{C}$ is holomorphic and $\Delta_{h_k}^m(f) =0$, $k=1,\cdots,s$, then  $f(z)=\sum_{|\alpha|<N}a_{\alpha}z^{\alpha}$ is an ordinary complex-valued polynomial in $k$ complex variables.
\end{theorem}

The finitely generated subgroups of  $(\mathbb{R}^d,+)$ which are dense in $\mathbb{R}^d$ have been actively studied and, in fact, they can be characterized in several ways. For example, in \cite[Proposition 4.3]{W}, the following theorem is proved:
\begin{theorem}\label{subgruposdensos} Let $G=h_1\mathbb{Z}+h_2\mathbb{Z}+\cdots+h_{s}\mathbb{Z}$ be the additive subgroup of  $\mathbb{R}^d$ generated by the vectors $\{h_1,\cdots,h_{s}\}$. The following statements are equivalent:
\begin{itemize}
\item[$(i)$] $G$ is a dense subgroup of $\mathbb{R}^d$.
\item[$(ii)$] If $h_k=(a_{1k},a_{2k},\cdots,a_{dk})$ are the coordinates of $h_k$ with respect to the canonical basis of   $\mathbb{R}^d$ $(k=1,\cdots, s)$, then the matrices
\[
A(n_1,\cdots,n_s)=\left [
\begin{array}{cccccc}
a_{11} & a_{12} &  \cdots & a_{1s} \\
a_{21} & a_{22} & \cdots & a_{2s}\\
\vdots & \vdots & \ \ddots & \vdots \\
a_{d1} & a_{d2} &  \cdots &  a_{ds}\\
n_1 & n_2 & \cdots &  n_{s}
\end{array} \right].
\]
have rank equal to $d+1$, for all  $(n_1,\cdots,n_s)\in\mathbb{Z}^s\setminus\{(0,\cdots,0)\}$.
\end{itemize}
\end{theorem}

A simple case which has motivated the study of dense subgroups of  $(\mathbb{R}^d,+)$, is the following one:

\begin{corollary}[Kronecker's theorem] \label{Kro} Given $\theta_1,\theta_2,\cdots,\theta_d\in\mathbb{R}$, the group  $\mathbb{Z}^d+(\theta_1,\theta_2,\cdots,\theta_d)\mathbb{Z}$ (which is generated by exactly $d+1$ elements) is a dense subgroup of   $\mathbb{R}^d$ if and only if
\[
n_1\theta_1+\cdots+n_d\theta_d\not\in \mathbb{Z}, \text{ for all  }(n_1,\cdots,n_d)\in\mathbb{Z}^{d}\setminus \{(0,\cdots,0)\}
\]
in other words, this group is dense in  $\mathbb{R}^d$ if and only if the numbers $\{1,\theta_1,\cdots,\theta_d\}$ form a linearly independent system  over $\mathbb{Q}$.
\end{corollary}

\begin{proof} The vectors $\{e_k\}_{k=1}^d\cup\{(\theta_1,\cdots,\theta_d)\}$, where $$e_k=(0,0,\cdots,1^{(k\text{-th position})},0,\cdots,0), \ \ k=1,\cdots, d,$$ generate the  group $G=\mathbb{Z}^d+(\theta_1,\theta_2,\cdots,\theta_d)\mathbb{Z}$. Thus, part $(ii)$ from  Theorem \ref{subgruposdensos} guarantees that $G$ is dense in $\mathbb{R}^d$
if and only if
\begin{eqnarray*}
&\ & \det\left [
\begin{array}{cccccc}
\theta_{1} & 1 & 0 &  \cdots  & 0 \\
\theta_{2} & 0 & 1 &   \cdots  & 0 \\
\vdots & \vdots & \ \ddots & \vdots \\
\theta_{d} & 0 & 0&  \cdots  & 1 \\
n_0 & n_1 & n_2 &  \cdots &  n_{d}
\end{array} \right]  \\
&=&  (-1)^{d+2}n_0+\sum_{k=1}^d(-1)^{d+2+k}n_k(-1)^{k+1}\theta_k \det(I_{d-1}) \\
  &=&  (-1)^{d}\left(n_0-\sum_{k=1}^dn_k\theta_k\right) \neq 0
\end{eqnarray*}
for all $(n_0,n_1,\cdots,n_d)\in\mathbb{Z}^{d+1}\setminus \{(0,\cdots,0)\}$, which is equivalent to claim that
\[
n_1\theta_1+\cdots+n_d\theta_d\not\in \mathbb{Z}, \text{ for all }(n_1,\cdots,n_d)\in\mathbb{Z}^{d}\setminus \{(0,\cdots,0)\},
\]
which is what we looked for. To prove the last claim in the theorem, it is enough to observe that, if  $\{1,\theta_1,\cdots,\theta_d\}$ forms a $\mathbb{Q}$-linearly dependent system, then  there are rational numbers   $r_i=n_i/m_i$, $i=0,1,\cdots, d$ (not all of them equal to zero), such that
\[
r_0+r_1\theta_1+\cdots+r_d\theta_d= 0,
\]
so that, multiplying both sides of the equation by  $m=\prod_{k=0}^dm_k$ we get
\[
n_0^*+n_1^*\theta_1+\cdots+n_d^*\theta_d= 0
\]
for certain natural numbers $n_0^*,n_1^*,\cdots,n_d^*$ (not all equal to zero). In particular,
\[
n_1^*\theta_1+\cdots+n_d^*\theta_d\in\mathbb{Z}.
\]
\end{proof}

Montel's theorem was improved by his student T. Popoviciu \cite{popoviciu}  for the case $d=1$ to the following:

\begin{theorem}[Popoviciu, 1935]  \label{Popov} Let $f:\mathbb{R}\to\mathbb{R}$ be such that
\[
\Delta^{m+1}_{h_1}f(x)=\Delta^{m+1}_{h_2}f(x)=0,  \text{  } (x\in\mathbb{R}),
\]
If $h_1/h_2\not\in\mathbb{Q}$ and $f$ is continuous in at least  $m+1$ distinct points, then $f\in \Pi_m$.
\end{theorem}

In section 2 we prove, by very elementary means, a theorem which generalizes Montel's Theorem in several variables.  In section 3 we prove a version of Montel-Popoviciu's theorem for functions $f:\mathbb{R}^d\to\mathbb{R}$ for $d>1$. Furthermore, our proof is also valid for the case $d=1$, and, in that case, it differs in several points from Popoviciu's original proof.  In section 4 we prove that Popoviciu's original result is optimal, since, if $h_1,h_2\in\mathbb{R}$ and $h_1/h_2\not\in\mathbb{Q}$, there exists $f:\mathbb{R}\to\mathbb{R}$ continuous in $m$ distinct points which is not an ordinary polynomial and solves the system of equations $\Delta^{m+1}_{h_1}f(x)=\Delta^{m+1}_{h_2}f(x)=0$. Finally, in that section we also consider the optimality of Montel-Popoviciu's theorem in the several variables setting.




In this paper we use the following standard notation: $\Pi_m$ denotes the space of complex polynomials of a real variable, with degree $\leq m$. $\Pi_{m,\max}^d$ denotes the space of complex polynomials of $d$ real variables, with degree $\leq m$ with respect to each one of these variables. More precisely,
\[
P(x_1,\cdots,x_d)\in \Pi_{m,\max}^d \text{ if and only if } P(x_1,\cdots,x_d)=\sum_{i_1=0}^m\sum_{i_2=0}^m\cdots\sum_{i_d=0}^ma_{i_1,\cdots,i_d}x_1^{i_1}\cdots x_d^{i_d},
\]
with $a_{i_1,\cdots,i_d}$ being complex numbers. Finally, $\Pi_{m,\text{tot}}^d$ denotes the space of complex polynomials of $d$ real variables, with total degree $\leq m$; that is,
\[
P(x_1,\cdots,x_d)\in \Pi_{m,\text{tot}}^d \text{ if and only if } P(x_1,\cdots,x_d)=\sum_{i_1,\cdots,i_d\in\mathbb{N} \text{ and } i_1+\cdots+i_d\leq m} a_{i_1,\cdots,i_d}x_1^{i_1}\cdots x_d^{i_d},
\]
with $a_{i_1,\cdots,i_d}$ being complex numbers. Of course, $\Pi_m=\Pi_{m,\max}^1=\Pi_{m,\text{tot}}^1$.

\section{A generalization of Montel theorem}

For the statement of the results in this section, we need to recall the following concept from interpolation theory: Given $\Lambda\subseteq \Pi^d$ a subspace of the space of polynomials in $d$ real variables, and given $W\subseteq \mathbb{R}^d$, we say that $W$ is a correct interpolation set  (sometimes  also referred as insolvent ) for $\Lambda$ if and only if  for any function $f:W\to\mathbb{C}$ there exists a unique polynomial $P\in \Lambda$ such that $P(x)=f(x)$ for all $x\in W$. In particular, if $W$ is a correct interpolation set for $\Lambda$ and $P\in\Lambda$ satisfies $P_{|W}=0$, then $P=0$. These sets have been characterized for several spaces of polynomials $\Lambda$ \cite{tesis_correct_interp}. In particular, the technique of tensor product interpolation guarantees that the sets $W$ of the form
$$W=\{x_0^1,x_1^1,\cdots,x_m^1\}\times \{x_0^2,x_1^2,\cdots,x_m^2\}\times \cdots \times \{x_0^d,x_1^d,\cdots,x_m^d\},$$
which form a rectangular grid of $(m+1)^d$ points in $\mathbb{R}^d$, are correct interpolation sets for $\Lambda=\Pi_{m,\max}^d$ (see, e.g., \cite[page 295]{I_K}).

Let $G$ be an Abelian group, let  $f:G\to\mathbb{C}$ be an arbitrary function. The  functions $\Delta_h^nf:G\to\mathbb{C}$ are well defined and $f$ is named a complex polynomial function of degree $\leq m$  on $G$ if  $\Delta_h^{m+1}f(x)=0$ for all $x,h\in G$. If there exist additive functions $a_k:G\to \mathbb{C}$, $k=1,\cdots,t$ and an ordinary complex polynomial $P\in\mathbb{C}[x_1,\cdots,x_t] $ of total degree $\leq m$ such that $f(x)=P(a_1(x),\cdots,a_t(x))$ for all $x\in G$, we just say that $f$ is a polynomial on $G$. It is known that, for finitely generated Abelian groups, every polynomial function is a polynomial \cite{laszlo1}.

Given $\gamma=\{h_1,\cdots,h_s\}\subseteq G$, $a\in G$, and given $f:G\to \mathbb{C}$ a polynomial function of degree $\leq m$,  there exists a unique polynomial $P=P_{a,\gamma}\in\Pi_{m,\max}^s$ such that  
\[
P(i_1,i_2,\cdots,i_{s})=f_{i_1,\cdots,i_{s}}:=f(a+\sum_{k=1}^{s}i_kh_k),
\]
for all  $0\leq i_k\leq m$, $1\leq k\leq s$, since $W=\{(i_1,\cdots,i_s): 0\leq i_k\leq m,\ k=1,\cdots,s\}$ is a correct interpolation set for $\Pi_{m,\max}^s$. In all what follows, we denote this polynomial  by  $P_{a,\gamma}$.


\begin{lemma} \label{MP_lem1} Let $G$ be a commutative group and $f:G\to\mathbb{C}$ be a function. If 
$\Delta_{h_k}^{m+1}f(x)=0 $  for all $x \in G $ and $k=1,\cdots, s$,  then
\[
P_{a,\gamma}(i_1,i_2,\cdots,i_{s})=f(a+\sum_{k=1}^{s}i_kh_k), \text{ for all  } (i_1,\cdots,i_{s})\in\mathbb{Z}^{d+1}.
\]
\end{lemma}

\noindent \textbf{Proof. } Let us fix the values of  $k\in\{1,\cdots,s\}$ and $i_1,\cdots,i_{k-1},i_{k+1},\cdots,i_{s}\in\{0,1,\cdots,m\}$, and let us consider the polynomial of one variable
$$q_k(x)=P_{a,\gamma}(i_1,\cdots, i_{k-1},x,i_{k+1}, \cdots,i_{d+1}) .$$
Obviously $q_k\in\Pi_m^1$, so that
\begin{eqnarray*}
&\ & 0 = \Delta_{1}^{m+1}q_k(0)=\sum_{r=0}^{m+1}\binom{m+1}{r}(-1)^{m+1-r}q_k(r) \\
&\ & = \sum_{r=0}^{m}\binom{m+1}{r}(-1)^{m+1-r}P_{a,\gamma}(i_1,\cdots, i_{k-1},r,i_{k+1}, \cdots,i_{d+1}) +  q_k(m+1)\\
&\ & = \sum_{r=0}^{m}\binom{m+1}{r}(-1)^{m+1-r}f(a+\sum_{(0\leq j\leq s;\ j\neq k)} i_jh_j+ rh_k) + q_k(m+1)\\
&\ & = \Delta_{h_k}^{m+1}f(a+\sum_{(0\leq j\leq s;\ j\neq k)} i_jh_j) -f(a+\sum_{(0\leq j\leq s;\ j\neq k)} i_jh_j+ (m+1)h_k)\\
&\ & \ \ \ + q_k(m+1)\\
&\ & =  q_k(m+1)-f(a+\sum_{(0\leq j\leq s;\ j\neq k)} i_jh_j+ (m+1)h_k).
\end{eqnarray*}

It follows that
\begin{equation} \label{nuevodato}
 \begin{array}{cccccccc}
q_k(m+1) & = & \  &   P_{a,\gamma}(i_1,\cdots, i_{k-1},(m+1),i_{k+1}, \cdots,i_{s}) \\
\  & = & \  &  f(a+\sum_{(0\leq j\leq s;\ j\neq k)} i_jh_j+ (m+1)h_k).
\end{array}
\end{equation}
Let us now consider the unique  polynomial $P\in\Pi_{m,max}^{s}$ which satisfies the Lagrange interpolation conditions
\[
P(i_1,i_2,\cdots,i_{s}h_{s}) = f(a+\sum_{k=1}^{s}i_kh_k)
\]
for all $0\leq i_j\leq m$,  $1\leq j\leq s, j\neq k$, and all $1\leq i_k\leq m+1$.  We have already demonstrated, with formula \eqref{nuevodato}, that this polynomial  coincides with $P_{a,\gamma}$. Furthermore, the very same arguments used to prove  \eqref{nuevodato}, applied to the polynomial $P=P_{a,\gamma}$, lead us to the conclusion that
\begin{eqnarray*}
&\ &  P_{a,\gamma}(i_1,\cdots, i_{k-1},(m+2),i_{k+1}, \cdots,i_{s}h_{s}) \\
 &\ & \ \ \ =  f(a+\sum_{(0\leq j\leq s;\ j\neq k)} i_jh_j+ (m+2)h_k)
\end{eqnarray*}
In an analogous way, extracting this time the first term of the sum, and taking as starting point the equality
\[
 \Delta_{h_k}^{m+1}f(a+\sum_{(0\leq j\leq s;\ j\neq k)} i_jh_j-h_k)=0,
\]
we conclude that
\begin{eqnarray*}
 &\ & P_{a,\gamma}(i_1,\cdots, i_{k-1},-1,i_{k+1}, \cdots,i_{s}) \\
 &\ & \ \ \ = f(a+\sum_{(0\leq j\leq s;\ j\neq k)} i_jh_j- h_k).
\end{eqnarray*}
Repeating these arguments forward and backward infinitely many times, and for each  $k\in\{1,\cdots,s\}$, we get
\begin{eqnarray*}
&\ & P_{a,\gamma}(i_1,i_2,\cdots,i_{s})\\
&\ & \ \ \  = f(a+\sum_{k=1}^{s}i_kh_k), \text{ for all } (i_1,\cdots,i_{s})\in\mathbb{Z}^{d+1},
\end{eqnarray*}
which is what we wanted to prove. {\hfill $\Box$}

\begin{corollary} \label{nuevo_montel_0}  Let $G$ be  a topological Abelian group. If $f:G\to\mathbb{C}$ satisfies
$\Delta_{h_k}^{m+1}f(x)=0 $  for all $x \in G $ and $k=1,\cdots, s$,  then
\[
\Delta_h^{sm+1}f(x)=0
\]
for all $h\in H=h_1\mathbb{Z}+\cdots+h_s\mathbb{Z}$. 
\end{corollary}

\begin{proof}
Let $P_{a,\gamma}$ be the polynomial constructed in Lemma \ref{MP_lem1}, and let $x\in\mathbb{R}^d$. Then
\[
P_{x,\gamma}(i_1,i_2,\cdots,i_{s})=f(x+\sum_{k=1}^{s}i_kh_k), \text{ for all  } (i_1,\cdots,i_{s})\in\mathbb{Z}^{d+1}.
\]
Hence, if $h=\sum_{k=1}^{s}i_kh_k\in G$,
\begin{eqnarray*}
\Delta_h^{sm+1}f(x) &=& \sum_{r=0}^{sm+1}\binom{sm+1}{r}(-1)^{sm+1-r}f(x+r\sum_{k=1}^{s}i_kh_k)\\
&=& \sum_{r=0}^{sm+1}\binom{sm+1}{r}(-1)^{sm+1-r}P_{x,\gamma}(r(i_1,\cdots,i_s))\\
&=& \Delta_{(i_1,\cdots,i_s)}^{sm+1}P_{x,\gamma}(0) =0,
\end{eqnarray*}
since $P_{x,\gamma}$ is a polynomial in $s$ variables of total degree $\leq sm$. This proves the first part of the corollary.
\end{proof}

\begin{corollary} \label{nuevo_montel}  If $G$ is a topological Abelian group, $f:G \to\mathbb{R}$ is continuous and satisfies
$\Delta_{h_k}^{m+1}f(x)=0 $  for all $x \in G $ and $k=1,\cdots, s$ and $ H=h_1\mathbb{Z}+\cdots+h_s\mathbb{Z}$ is a dense subgroup of $G$, then
\[
\Delta_h^{sm+1}f(x)=0
\]
for all $x,h\in G$. In other words, $f$ is a continuous polynomial on $G$ of degree $\leq sm$. 

In particular, if $G=\mathbb{R}^d$, then $f\in \Pi_{sm,\text{tot}}^d$.
\end{corollary}

\begin{proof}
Let us now assume also that $f$ is continuous and $G$ is dense in $\mathbb{R}^d$. Then it is clear that $\Delta_h^{sm+1}f(x) =0$ for all $x,h$ in $\mathbb{R}^d$, and, if $G=\mathbb{R}^d$, Fr\'{e}chet's Theorem  implies that $f$ is an ordinary polynomial. In fact, it is not difficult to prove that, in this case, $f$ must have total degree $\leq sm$ (see, for example, \cite[Theorem 3.1]{AK_TP}).
\end{proof}

\begin{remark} Corollary \ref{nuevo_montel} is optimal for $G=\mathbb{R}^d$ since, for each $s\geq d+1$ there exists $f:\mathbb{R}^d\to\mathbb{R}$ and 
$\{h_1,\cdots,h_s\}\subset \mathbb{R}^d$ such that $H=h_1\mathbb{Z}+\cdots+h_s\mathbb{Z}$ is dense in $\mathbb{R}^d$, $\Delta_{h_k}^{m+1}f(x)=0 $  for all $x \in\mathbb{R}^d $ and $k=1,\cdots, s$, and $\Delta_{h}^{sm}f\neq 0$ for some $h\in H$.  To prove this, let $\gamma=\{h_1,\cdots,h_s\}\subseteq \mathbb{R}^d$ be such that $G=h_1\mathbb{Z}+\cdots+h_s\mathbb{Z}$ is dense in $\mathbb{R}^d$, and let us assume that
\[
A_{\gamma}=\textbf{col}[h_1,\cdots,h_s] =\left [\begin{array}{ccccccccc}
a_{11} & a_{12} & \cdots & a_{1s}\\
a_{21} & a_{22} & \cdots & a_{2s}\\
\vdots & \ddots & \cdots & \vdots\\
a_{d1} & a_{d2} & \cdots & a_{ds}\\
\end{array}
\right]
\]
contains a set $\{a_{k1} , a_{k2},  \cdots,  a_{ks}\}$ which forms a $\mathbb{Q}$-linearly independent set of real numbers. If we consider two vectors $(i_1,\cdots,i_s),(j_1,\cdots,j_s)\in\mathbb{Z}^s$, then
\[
i_1h_1+\cdots+i_sh_s= j_1h_1+\cdots+j_sh_s \text{ if and only if } i_k=j_k \text{ for all } k=1,\cdots,s.
\]
Hence the function $f:\mathbb{R}^d\to\mathbb{R}$ given by
\[
f(x)=\left\{\begin{array}{cccc}
P(i_1,\cdots,i_s) & & \text{ if } x = i_1h_1+\cdots+i_sh_s \text{ and } (i_1,\cdots,i_s)\in\mathbb{Z}^s\\
0 & & \text{otherwise}
\end{array}\right.,
\]
is well defined for any map $P:\mathbb{Z}^s\to\mathbb{R}$.  Let  $P\in \Pi_{m,\max}^s\subseteq \Pi_{sm,\text{tot}}^s$  be the polynomial $P(x_1,\cdots,x_s)=x_1^mx_2^m\cdots x_s^m$. 
Then 
\[
\Delta_{h_k}^{m+1}f(x)=0,\ \text{ for all } x\in \mathbb{R} \text{ and } k=1,\cdots,s,
\]
and, on the other hand, if we consider the monomial $\varphi_N:\mathbb{R}\to\mathbb{R}$ defined by $\varphi_N(t)=t^N$, then 
\begin{eqnarray*}
\Delta_{h_1+\cdots+h_s}^{sm}f(0) &=& \Delta_{(1,1,1,\cdots,1)}^{sm}P(0,0,\cdots,0)= \sum_{k=0}^{sm} \binom{sm}{k}(-1)^{sm-k} P(k,k,\cdots,k) \\
&=&  \sum_{k=0}^{sm} \binom{sm}{k}(-1)^{sm-k} k^{sm}\\
&=& \Delta_{1}^{sm}\varphi_{sm}(0)= (sm)!\varphi_{sm}(1)=(sm)!\neq 0,
\end{eqnarray*}
since the monomials $\varphi_N$ satisfy the functional equation $\frac{1}{N!}\Delta_h^N\varphi_N(x)=\varphi_N(h)$. Hence $f$ satisfies our requirements.

To support our argument let us show an example of  matrix $A_{\gamma}$ satisfying our hypotheses. Let $\{1,\theta_1,\cdots,\theta_d\}$ be a $\mathbb{Q}$-linearly independent set of real numbers of size $d+1$ with $s=d+1$. Indeed, we impose  $\theta_k=\pi^k$ for all $k$.  Let us consider  the matrix
\[
A_{\gamma}=\textbf{col}[h_1,\cdots,h_s] :=\left [
\begin{array}{cccccc}
\pi & 1 & \pi^{2} &  \cdots  & \pi^{d} \\
\pi^{2} & 0 & 1 &   \cdots  & 0 \\
\vdots & \vdots & \ \ddots & \vdots \\
\pi^{d} & 0 & 0&  \cdots  & 1
\end{array} \right]
\]
Then part $(ii)$ of Theorem \ref{subgruposdensos} claims that $G=h_1\mathbb{Z}+\cdots+h_s\mathbb{Z}$ is dense in $\mathbb{R}^d$ if and only if, for each $(n_0,\cdots,n_d)\in\mathbb{Z}^{d+1}\setminus\{(0,\cdots,0)\}$,
\[
0\neq \det(B_{\gamma}(n_0,\cdots,n_d)):=\det \left [
\begin{array}{cccccc}
\pi & 1 & \pi^{2} &  \cdots  & \pi^{d} \\
\pi^{2} & 0 & 1 &   \cdots  & 0 \\
\vdots & \vdots & \ \ddots & \vdots \\
\pi^{d} & 0 & 0&  \cdots  & 1 \\
n_0 & n_1 & n_2 &  \cdots &  n_{d}
\end{array} \right]
\]
Now, a simple computation (expanding the determinant by the first row) that
\begin{eqnarray*}
\det(B_{\gamma}(n_0,\cdots,n_d)) &=& (-1)^dn_0+(-1)^{d+1}[\pi-\sum_{k=2}^d\pi^{2k}]n_1+(-1)^{d+1}\sum_{k=2}^d\pi^kn_k\\
&=& (-1)^dn_0+(-1)^{d+1}\sum_{k=1}^d\pi^kn_k +(-1)^d\sum_{k=2}^d\pi^{2k}n_1,
\end{eqnarray*}
which does not vanish if  $(n_0,\cdots,n_d)\in\mathbb{Z}^{d+1}\setminus\{(0,\cdots,0)\}$, since $\pi$ is a transcendental number. This solves the case $s=d+1$. The general case follows as a direct consequence of this one, since the matrices
\[
\left[\begin{array}{cccccccccc}
\pi & 1 & \pi^{2} &  \cdots  & \pi^{d} &\pi^{d+1} & \cdots &\pi^s\\
\pi^{2} & 0 & 1 &   \cdots  & 0 & 0 & \cdots & 0\\
\vdots & \vdots & \ \ddots & \vdots & \vdots &\vdots &\ddots & \vdots \\
\pi^{d} & 0 & 0&  \cdots  & 1& 0 & \cdots & 0 \\
n_0 & n_1 & n_2 &  \cdots &  n_{d}& n_{d+1} &\cdots & n_s
\end{array} \right]
\]
contain $B_{\gamma}(n_0,\cdots,n_d)$ as a submatrix and, hence, have maximal rank for all  $s>d$ and all  $(n_0,\cdots,n_s)\in\mathbb{Z}^{s+1}\setminus\{(0,\cdots,0)\}$.


 \end{remark}

Several distributional techniques have been used in the study of functional equations. In particular, Fourier transform of tempered distributions is used in \cite{baker} to introduce a method of solving some special functional equations. This motivates  study  Montel's type theorems in the distributional setting. 
Let $\mathcal{S}'(\mathbb{R}^d)$ denote the space of complex tempered distributions defined on $\mathbb{R}^d$. Obviously, $\mathcal{S}'(\mathbb{R}^d)$ is a vector space and  the operators $\Delta_h^s$ can be defined as endomorphisms of $\mathcal{S}'(\mathbb{R}^d)$ by the formula, $$\Delta_h^sf\{\varphi\}=(-1)^sf\{\Delta_{-h}^s\varphi\}$$ for $s=1,2,\cdots$. Furthermore, if $\alpha=(\alpha_1,\cdots,\alpha_d)\in\mathbb{N}^d$ is any multi-index and we denote by \\
$D^\alpha f=\displaystyle \frac{\partial ^{\alpha_1}}{\partial x_1^{\alpha_1}} \frac{\partial ^{\alpha_2}}{\partial x_2^{\alpha_2}}\cdots \frac{\partial ^{\alpha_d}}{\partial x_d^{\alpha_d}}f$ the $\alpha$-th generalized derivative of $f$, then $\Delta_h^mD^\alpha=D^\alpha\Delta_{h}^m$ for all $m\in\mathbb{N}$, and if $D^\alpha f=0$ for all multi indices $\alpha$ with $|\alpha|=n$, then there exists  $P\in\Pi_{n-1,\text{tot}}^d$ such that $P=f$ almost everywhere. Of course, if $f$  is taken to be continuous, then $f=P$ everywhere. Finally, the structure theorem for tempered distributions guarantees that if $f\in \mathcal{S}'(\mathbb{R}^d)$ then there exist a  slowly growing continuous function $F$  and  a natural number $n\in\mathbb{N}$ such that $f= D^{(n,n,\cdots,n)}F$ \cite[page 98]{vla}. Tempered distributions are also interesting because the Fourier transform $\mathcal{F}$ can be defined as an automorphism $\mathcal{F}:\mathcal{S}'(\mathbb{R}^d)\to\mathcal{S}'(\mathbb{R}^d)$ just imposing $\mathcal{F}(f)\{\varphi\}=f\{\mathcal{F}(\varphi)\}$, and the new operator preserves the main properties of the classical Fourier transform (see, for example, \cite[page 144]{donoghue}, \cite[page 192, Theorem 7.15]{rudin}).

\begin{lemma} \label{operator}
Assume that $f\in \mathcal{S}'(\mathbb{R}^d)$ and let $\Gamma=\Delta_{e_1}^1\Delta_{e_2}^1\cdots\Delta_{e_d}^1$ and $n\in\mathbb{N}\setminus\{0\}$.  Then  $D^{(n,n,\cdots,n)}f=0$ implies $\Gamma^nf=0$.
\end{lemma}
\begin{proof}
$f\in \mathcal{S}'(\mathbb{R}^d)$ implies that $$\mathcal{F}(D^{\alpha}f)(\xi)= (i\xi)^{\alpha}\mathcal{F}(f)(\xi)$$ for every multi-index $\alpha$ and 
$$\mathcal{F}(\Delta_{h_0}f(x))(\xi)=\mathcal{F}(f(x+h_0)-f(x))(\xi)= (e^{i<h_0,\xi>}-1)\mathcal{F}(f)(\xi)$$ for every step $h_0\in\mathbb{R}^d$. Hence, if $\mathcal{D}=D^{(1,1,\cdots,1)}$ and $\mathcal{D}(f)=0$, then  
\begin{equation} \label{Fourier}
0=\mathcal{F}(\mathcal{D}(f)) (\xi)= (i \xi)^{(1,\cdots,1)} \mathcal{F}(f)(\xi)= (i)^d\xi_1\xi_2\cdots \xi_d  \mathcal{F}(f)(\xi).
\end{equation}
In particular, the support of $\mathcal{F}(f)$ is a subset of $V=\bigcup_{k=1}^d H_k$, where $H_k=\{\xi:  \xi_k=0\}$ is an hyperplane of $\mathbb{R}^d$ for every $k$.  
On the other hand, 
\begin{equation} \label{despla}
\mathcal{F}(\Gamma(f))(\xi) =\left(\prod_{k=1}^d(e^{i<e_k,\xi>}-1)\right)\mathcal{F}(f)(\xi) =\left( \prod_{k=1}^d(e^{i\xi_k}-1)\right)\mathcal{F}(f)(\xi) 
\end{equation}
and, evidently, if the support of $\mathcal{F}(f)$ is a subset of $V$, then $\prod_{k=1}^d(e^{i\xi_k}-1)\mathcal{F}(f)(\xi) =0$ (since $\mathcal{F}(f)$ vanishes on all points $\xi$  such that $\prod_{k=1}^d(e^{i\xi_k}-1)\neq 0$) and  $\mathcal{F}(\Gamma(f))(\xi) =0$, which implies $\Gamma(f)=0$. Thus, if $f\in \mathcal{S}'(\mathbb{R}^d)$ and  
$\mathcal{D}(f)=0 $, then $\Gamma(f)=0$, which is the case $n=1$ of the lemma. 
Assume the result  holds for $n-1$ and let $f\in \mathcal{S}'(\mathbb{R}^d)$ be such that 
$\mathcal{D}^nf=0$. Then 
\[
0=\mathcal{D}^nf= \mathcal{D}(\mathcal{D}^{n-1}f),
\]
so that 
\[
0=\Gamma(\mathcal{D}^{n-1}f) =\mathcal{D}^{n-1}(\Gamma f) 
\]
and the induction hypothesis implies that $\Gamma^{n-1}(\Gamma f)=0$, which is what we wanted to prove.
\end{proof}

We are now able to prove the following result:

\begin{corollary} \label{distributions}
Let $f\in \mathcal{S}'(\mathbb{R}^d)$ and let $H=h_1\mathbb{Z}+\cdots+h_s\mathbb{Z}$ be a dense subgroup of $\mathbb{R}^d$. Assume that $\Delta_{h_k}^{m+1}f=0$, $k=1,\cdots,s$. Then
there exists $f^*\in\Pi_{sm,\text{tot}}^d$ such that $f=f^*$ almost everywhere.
\end{corollary}

\begin{proof}
Let $f\in \mathcal{S}'(\mathbb{R}^d)$ satisfy the hypotheses of this corollary and let us take $n\in\mathbb{N}$ and $F:\mathbb{R}^d\to\mathbb{C}$ a continuous slowly growing function such that  $f=D^{(n,n,\cdots,n)}F$. Then, for $1\leq k\leq s$, we have that
\begin{eqnarray*}
0 &=& \Delta_{h_k}^{m+1}f \\
&=&  \Delta_{h_k}^{m+1}D^{(n,n,\cdots,n)} F=  D^{(n,n,\cdots,n)}(\Delta_{h_k}^{m+1}F)
\end{eqnarray*}
Hence, Lemma \ref{operator} implies that 
$$0=\Gamma^n(\Delta_{h_k}^{m+1}F)=\Delta_{h_k}^{m+1}(\Gamma^nF),\text{ for all } 1\leq k\leq s,$$
and since $\Gamma^n(F)$ is continuous, we can apply  Corollary \ref{nuevo_montel} to $\Gamma^n(F)$ to conclude that $\Gamma^n(F)$ is a polynomial. In particular, $F$ is of class $C^{(\infty)}$ and $f=D^{(n,n,\cdots,n)} F$  in distributional sense, which implies that $f$ is equal almost everywhere to a continuous function $f^*$ and this function $f^*$ satisfies $\Delta_{h_k}^{m+1}f^*=0$, $k=1,\cdots,s$ in the classical sense. Thus,  if we apply  Corollary \ref{nuevo_montel} to $f^*$ we conclude that $f^*\in\Pi_{sm,\text{tot}}^d$ and $f=f^*$ almost everywhere. This concludes the proof.

\end{proof}

Obviously, we can resume all results proved in this section with the statement of the following  generalized version of Montel's Theorem:

 \begin{theorem} \label{MG}
We suppose that $s$ is a positive integer, and either of the following possibilities holds:
\begin{enumerate}
\item $G$ is a finitely generated Abelian group with generators $h_1,\dots,h_s$, and \hbox{$f:G\to\mathbb{C}$}  is a function.
\item $G$ is a topological Abelian group, in which the elements $h_1,\dots,h_s$ generate a dense subgroup in $G$, and  $f:G\to\mathbb{C}$ is a continuous function.\end{enumerate}
If  $f$ satisfies
 \begin{equation} \label{para_dist}
 \Delta_{h_k}^{m+1}f=0
 \end{equation}
 for $k=1,2,\dots,s$, then $f$ is a polynomial of total degree $\leq sm$ on $G$. Furthermore, if $G=\mathbb{R}^d$,  $sm$ is the best possible. Finally, if 
$G=\mathbb{R}^d$,  the elements $h_1,\dots,h_s$ generate a dense subgroup of $\mathbb{R}^d$, and $f$ is a complex valued tempered distribution on $\mathbb{R}^d$ which satisfies \eqref{para_dist}, then $f=p$ almost everywhere for some $p\in \Pi_{sm,\text{tot}}^d$. 
\end{theorem}

Note that, with completely different techniques, similar results have been recently demonstrated by Almira \cite{A_NFAO}, Almira-Abu Helaiel \cite{AK_CJM} and Almira-Sz\'{e}kelyhidi \cite{A_L_montel}. 

\section{Montel-Popoviciu theorem in several variables setting}

In this section we prove a result of Popoviciu's type for functions defined on the Euclidean space $\mathbb{R}^d$ for $d>1$. We begin by  a technical lemma showing that  every polynomial $P\in\Pi_{m,\max}^s$ can be decomposed as a special sum involving  polynomials of the form $A_{k}(t_1+\theta_1t_s,t_2+\theta_2t_s,\cdots,t_{s-1}+\theta_{s-1}t_s)$, with   $A_k\in \Pi_{(s-1)m,\max}^{s-1}$ and  $k=0,\cdots,sm$.

\begin{lemma}
Let $\{\theta_1,\cdots,\theta_{s-1}\}\subset \mathbb{R}\setminus \{0\}$. Then every polynomial $P\in\Pi_{m,\max}^s$ can be decomposed as a sum of the form
\[
 P(t_1,\cdots,t_s)=\sum_{k=0}^{sm}A_{k}(t_1+\theta_1t_s,t_2+\theta_2t_s,\cdots,t_{s-1}+\theta_{s-1}t_s)t_s^{k}
 \]
where $A_k\in \Pi_{(s-1)m,\max}^{s-1}$ for $k=0,\cdots,sm$.
\end{lemma}
\begin{proof}
Let $P(t_1,\cdots, t_s) =  \sum_{i_1=0}^m\sum_{i_2=0}^m\cdots \sum_{i_s=0}^ma_{i_1,i_2,\cdots,i_s} t_1^{i_1}\cdots t_s^{i_s}$ and let us consider   the change of variables given by  $f_1=t_1+\theta_1t_s,\cdots,f_{s-1}=t_{s-1}+\theta_{s-1}t_s$ and $f_s=t_s$. Then $t_k=f_k-\theta_kf_s$ for all $1\leq k\leq s-1$, and $f_s=t_s$, so that
\begin{eqnarray*}
P(t_1,\cdots, t_s) &= & \sum_{i_1=0}^m\sum_{i_2=0}^m\cdots \sum_{i_s=0}^ma_{i_1,i_2,\cdots,i_s}(f_1-\theta_1f_s)^{i_1}\cdots (f_{s-1}-\theta_{s-1}f_s)^{i_{s-1}} (f_s)^{i_s}\\
&=&  \sum_{k=0}^{sm}A_{k}(f_1,f_2,\cdots,f_{s-1})f_s^{k}\\
&=&  \sum_{k=0}^{sm}A_{k}(t_1+\theta_1t_s,t_2+\theta_2t_s,\cdots,t_{s-1}+\theta_{s-1}t_s)t_s^{k},
\end{eqnarray*}
where $A_k(f_1,\cdots,f_{s-1})$ is a polynomial of $s-1$ variables with degree at most $(s-1)m$ in each one of them, for $k=0,\cdots,sm$.
\end{proof}

\begin{theorem}[Characterization of polynomials of the form $A(t_1+\theta_1t_s,t_2+\theta_2t_s,\cdots,t_{s-1}+\theta_{s-1}t_s)$] \label{PoVV}
Let $\{\theta_1,\cdots,\theta_{s-1}\}\subset \mathbb{R}\setminus \{0\}$ and let $P(t_1,\cdots,t_s)\in\Pi_{m,\text{max}}^{s}$. Then
\[
P(t_1,\cdots,t_s)=A_{0}(t_1+\theta_1t_s,t_2+\theta_2t_s,\cdots,t_{s-1}+\theta_{s-1}t_s),
\]
with $A_0\in \Pi_{(s-1)m,\max}^{s-1}$ if and only if there exists $W\subseteq \mathbb{R}^{s-1}$, a correct interpolation set for $ \Pi_{(s-1)m,\max}^{s-1}$, such that,  for all $\alpha=(\alpha_{1},\cdots,\alpha_{s-1}) \in W$, there exists a sequence of vectors   $\{(u_{1,n},\cdots,u_{s-1,n}, u_{s,n})\}_{n=1}^{\infty}$ satisfying the following three conditions:
\begin{itemize}
\item[$(i)$] $u_{j,n}+\theta_ju_{s,n}\to \alpha_{j}$  when  $n\to\infty$, for  $1\leq j\leq s-1$.
\item[$(ii)$] $|u_{s,n}|\to\infty$ for $n\to\infty$, and
\item[$(iii)$] $\{P(u_{1,n},u_{2,n},\cdots,u_{s-1,n}, u_{s,n})\}_{n=1}^\infty$ is bounded.
\end{itemize}
\end{theorem}

To prove Theorem \ref{PoVV} we need first to state some technical results:

\begin{lemma} \label{MP_ceros_pol}
Let $p(z)=a_0+a_1z+\cdots+a_Nz^N\in \mathbb{C}[z]$ be an ordinary polynomial of degree $N$ (i.e., $a_N\neq 0$) and let  $\xi\in\mathbb{C}$ be a zero of $p$.Then
\[
|\xi|\leq \text{max}\{1,\sum_{k=0}^{N-1}\frac{|a_k|}{|a_N|}\}.
\]
\end{lemma}

\begin{proof}
This is a well known fact, but  we include the proof for the sake of completeness. If $|\xi|\leq 1$ we are done. Thus, let us assume $|\xi|>1$. Obviously, $q(z)=\frac{1}{|a_N|}p(z)= \sum_{k=0}^{N-1}\frac{a_k}{a_N}z^k+z^N$ satisfies $q(\xi)=0$. Hence
\begin{eqnarray*}
|\xi|^N=\left| \sum_{k=0}^{N-1}\frac{a_k}{a_N}\xi^k\right| \leq \sum_{k=0}^{N-1}\frac{|a_k|}{|a_N|}\max\{1,|\xi|,\cdots,|\xi|^{N-1}\} =  \left(\sum_{k=0}^{N-1}\frac{|a_k|}{|a_N|}\right)|\xi|^{N-1}.
\end{eqnarray*}
It follows that, in this case, $|\xi|\leq \sum_{k=0}^{N-1}\frac{|a_k|}{|a_N|}$, which is what we wanted to prove.
\end{proof}

\begin{lemma}\label{MP_poli_banda}
Let $p(z)=a_0+a_1z+\cdots+a_Nz^N\in \mathbb{C}[z]$ be an ordinary polynomial of degree $N$ (i.e., $a_N\neq 0$) and assume that $N\geq 1$.  Let $\{q_n(z)\}_{n=1}^\infty$ be a sequence of ordinary polynomials of degree $\leq N$, $$q_n(z)=a_{0n}+a_{1n}z+\cdots+a_{Nn}z^N,$$
and assume that $$\text{max}\{|a_k-a_{kn}|:k=0,1,\cdots,N\}<|a_N|/2,\ \ \  n=1,2,\cdots,\infty.$$
If  $|w_n|\to+\infty$, then  $|q_n(w_n)|\to\infty$.
\end{lemma}

\begin{proof} Let $n\in\mathbb{N}$ and let $\xi$ be a zero of $q_n(z)$. Then
\[
|\xi|\leq \text{max}\{1,\sum_{k=0}^{N-1}\frac{|a_{kn}|}{|a_{Nn}|}\},
\]
and, since $|a_{kn}|\leq |a_{kn}-a_k|+|a_k|\leq \frac{|a_N|}{2}+|a_k|$, $|a_{Nn}|\geq \frac{|a_N|}{2}$, we conclude that
\[
|\xi|\leq \text{max}\{1,\sum_{k=0}^{N-1}\frac{2(\frac{|a_N|}{2}+|a_k|)}{|a_{N}|}\}=:M.
\]
Thus, all zeroes of $q_n(z)$ belong to $B_M=\{z\in\mathbb{C}:|z|\leq M\}$ (for all $n$).

If $|w_n|\to\infty$, then $\mathbf{\text{dist}}(w_n,B_M)\to\infty$.

On the other hand, if $\{\alpha_{kn}\}_{k=1}^N$ denotes the set of zeroes of  $q_n(z)$, then $$q_n(z)=a_{Nn}\prod_{k=1}^N(z-\alpha_{kn}),$$
so that
\[
|q_n(w_n)|=|a_{Nn}|\prod_{k=1}^N|w_n-\alpha_{kn}|\geq \frac{|a_N|}{2}(\text{dist}(w_n,B_M))^N\to\infty. \ \ (n\to\infty).
\]
\end{proof}

\begin{proof}[Proof of Theorem \ref{PoVV}] The necessity is obvious, since every polynomial of the form
\[
P(t_1,\cdots,t_s) =  A_{0}(t_1+\theta_1t_s,t_2+\theta_2t_s,\cdots,t_{s-1}+\theta_{s-1}t_s),
\]
with $A_0\in\Pi_{(s-1)m,\max}^{s-1}$, is uniformly bounded on strips of the form $$\Gamma_{a,b}=\{(t_1,\cdots,t_s):a_k\leq t_k+\theta_kt_s\leq b_k,\ k=1,\cdots,s-1\},$$ where $\{a=(a_1,\cdots,a_{s-1}),b=(b_1,\cdots,b_{s-1})\}\subset \mathbb{R}^{s-1}$.

Let us now prove the sufficiency. Let $P(t_1,\cdots,t_s)\in \Pi_{m,\max}^s$. Then, for a certain $N\leq sm$, $P$ admits a representation of the form
$$P(t_1,\cdots,t_s) =  \sum_{k=0}^{N}A_{k}(t_1+\theta_1t_s,t_2+\theta_2t_s,\cdots,t_{s-1}+\theta_{s-1}t_s)t_s^{k},$$
where $A_k$ is a polynomial of $s-1$ variables with degree at most $(s-1)m$ in each one of them, for $0\leq k\leq N$, and $A_N\neq 0$.   We must prove $N=0$.

Assume, on the contrary, that  $N>0$.


By hypothesis, $A_N(\alpha)\neq 0$ for a certain $\alpha=(\alpha_1,\cdots,\alpha_{s-1})\in W$, since $A_N\in \Pi_{(s-1)m,max}^{s-1}\setminus\{0\}$ and $W$ is a correct interpolation set for $\Pi_{(s-1)m,max}^{s-1}$. Consider the polynomial
\[
p(z)=\sum_{i=0}^{N}A_i(\alpha)z^i
\]
Now, the functions  $A_i(t_1,\cdots,t_{s-1})$ are continuous and  $\{u_{j,n}+\theta_ju_{s,n}\}\to \alpha_j$ for
$n\to \infty$ and $1\leq j\leq s-1$, so that
$\{A_{i}(u_{1,n}+\theta_1u_{s,n},u_{2,n}+\theta_2u_{s,n},\cdots, u_{s-1,n}+\theta_{s-1}u_{s,n})\}\to A_i(\alpha)$ for $i\in \{0,1,\cdots, N\}$. Thus, we can assume with no loss of generality, that  $|A_{i}(u_{1,n}+\theta_1u_{s,n},u_{2,n}+\theta_2u_{s,n},\cdots, u_{s-1,n}+\theta_{s-1}u_{s,n})-A_i(\alpha)|<|A_N(\alpha)|/2$, $i=0,1,\cdots,N$,
$n\in\mathbb{N}$.

Hence, if we consider the sequence of polynomials
\[
q_n(z)=\sum_{i=0}^{N}A_{i}(u_{1,n}+\theta_1u_{s,n},u_{2,n}+\theta_2u_{s,n},\cdots, u_{s-1,n}+\theta_{s-1}u_{s,n})z^i, \ n=1,2,\cdots,
\]
then Lemma \ref{MP_poli_banda} implies that  $|q_n(u_{s,n})|\to\infty$ for $n\to\infty$. But $|q_n(u_{s,n})|=|P(u_{1,n},u_{2,n},\cdots,u_{s-1,n}, u_{s,n})|$ is bounded. Hence, if $N>0$ we get a contradiction. It follows that $N=0$ and   $P(t_1,\cdots,t_s) =  A_{0}(t_1+\theta_1t_s,t_2+\theta_2t_s,\cdots,t_{s-1}+\theta_{s-1}t_s)$, with $A_0\in\Pi_{(s-1)m,\max}^{s-1}$.  \end{proof}
In the following $\Delta_{(\theta_1,\cdots,\theta_d)}^{m+1}f(x)$ denotes the usual 

 $$\Delta_{(\theta_1,\cdots,\theta_d)}^{m+1}f(x) = \sum_{k=0}^{m+1} \binom{m+1}{k}(-1)^{m+1-k}f(x+k(\theta_1,\cdots,\theta_d)).$$


\begin{corollary} \label{cor_impo}
Assume that $f:\mathbb{R}^d\to\mathbb{R}$ is bounded on a certain open set $U\subseteq\mathbb{R}^d$, $U\neq \emptyset $, and
$H=\mathbb{Z}^d+(\theta_1,\theta_2,\cdots,\theta_d)\mathbb{Z}$ is a dense subgroup of   $\mathbb{R}^d$. If $\beta=\{e_k\}_{k=1}^d$ denotes the canonical basis of $\mathbb{R}^d$ and $f$ satisfies
\[
\Delta_{e_k}^{m+1}f(x)= 0, \text{ for } k=1,\cdots, d ; \,\,\,\,\,  \Delta_{(\theta_1,\cdots,\theta_d)}^{m+1}f(x)= 0 ,
\]
and $P_{a,\gamma}$ denotes the polynomial constructed in Lemma \ref{MP_lem1}, for $\gamma=\beta \cup \{(\theta_1,\cdots,\theta_d)\}$, then
$$P_{a,\gamma}(t_1,\cdots,t_{d+1}) =  A(t_1+\theta_1t_{d+1},t_2+\theta_2t_{d+1},\cdots,t_{d}+\theta_dt_{d+1}),$$ with $A\in\Pi_{dm,\max}^{d}$.
 \end{corollary}

 \begin{proof}
There is no loss of generality if we assume that $a=0$. Indeed, if we  use the notation $P_{a,\gamma}(f)$ for the polynomial constructed in  Lemma \ref{MP_lem1} for the function $f$, and we take $g(x)=f(x+a)$, then it is clear that $P_{a,\gamma}(f)=P_{\mathbf{0},\gamma}(g)$.

Let $W\subseteq U$ be a correct interpolation set for $\Pi_{dm,\max}^{d}$ whose entries have rational coordinates (such set obviously exists since $U$ is open). Then, for every $\alpha=(\alpha_1,\cdots,\alpha_d)\in W$ there exists a sequence $\{(i_{1,n},\cdots,i_{d,n})+i_{d+1,n}(\theta_1,\cdots,\theta_d) \}_{n=1}^{\infty}$ which is contained in $U$ and satisfies
 \[
 \lim_{n\to\infty}(i_{1,n},\cdots,i_{d,n})+i_{d+1,n}(\theta_1,\cdots,\theta_d) =\alpha.
 \]
 The density of $H=\mathbb{Z}^d+(\theta_1,\theta_2,\cdots,\theta_d)\mathbb{Z}$ in   $\mathbb{R}^d$ implies that the numbers $\{1,\theta_1,\cdots,\theta_d\}$ form a linearly independent system  over $\mathbb{Q}$. In particular, $\theta_k$ is an irrational number for $k=1,\cdots,d$. From this, and from the convergence of $i_{k,n}+\theta_k i_{d+1,n}$ to $\alpha_k \in\mathbb{Q}$, it follows that  $|i_{d+1,n}|\to\infty$ for $n\to\infty$. On the other hand,
 \[
 P_{0,\gamma}(i_{1,n},\cdots,i_{d,n},i_{d+1,n}) = f((i_{1,n},\cdots,i_{d,n})+i_{d+1,n}(\theta_1,\cdots,\theta_d))
 \]
 is bounded. Thus, we can apply Theorem \ref{PoVV} to $P_{0,\gamma}$, completing the proof.
 \end{proof}

 Now we are ready to prove the main result of this section:

 \begin{theorem}[Montel-Popoviciu theorem for several variables]\label{M_T}
Assume that
$H=\mathbb{Z}^d+(\theta_1,\theta_2,\cdots,\theta_d)\mathbb{Z}$ is a dense subgroup of   $\mathbb{R}^d$. If $\beta=\{e_k\}_{k=1}^d$ denotes the canonical basis of $\mathbb{R}^d$ and $f:\mathbb{R}^d \to\mathbb{R}$ satisfies
\[
\Delta_{e_k}^{m+1}f(x)= 0, \text{ for } k=1,\cdots, d, \text{ and }\,\,\, \Delta_{(\theta_1,\cdots,\theta_d)}^{m+1}f(x)= 0,
\]
and is continuous at every point of a set  $W\subseteq\mathbb{R}^d$ which is a correct interpolation set for  $\Pi_{dm,\max}^{d}$,
then $f\in  \Pi_{m,\max}^{d}$.
 \end{theorem}
 \begin{proof} We divide the proof into two parts. In the first one we prove that $f$ is an ordinary polynomial which belongs to $\Pi_{dm,\max}^{d}$. In the second part we improve the result by showing that $f\in   \Pi_{m,\max}^{d}$, which is a smaller space of polynomials.

 \noindent \textbf{Part I:} To show that $f\in  \Pi_{dm,\max}^{d}$ , note that continuity of $f$ at just one point implies that $f$  is bounded on a certain nonempty open set $U\subseteq\mathbb{R}^d$, so that we can apply Corollary \ref{cor_impo} to $f$. In particular,  for every $a=(a_1,\cdots,a_d)\in\mathbb{R}^d$, there exists $A_a\in \Pi_{dm,\max}^{d}$ satisfying
 \[
 A_a((n_1,\cdots,n_d)+n_{d+1}(\theta_1,\cdots,\theta_d))=f(a+(n_1,\cdots,n_d)+n_{d+1}(\theta_1,\cdots,\theta_d)) \text{ for all } (n_1,\cdots,n_{d+1})\in\mathbb{Z}^{d+1}.
 \]
 The result follows if we prove that, for any $a\in\mathbb{R}^d$, the relation
 \begin{equation}
 \label{igualdad} A_a(x)=A_{\mathbf{0}}(x-a)
\end{equation}
holds for all $x\in\mathbb{R}^d$ (here, $\mathbf{0}$ denotes the zero vector $(0,\cdots,0)\in\mathbb{R}^d$). To prove this, it is enough to take into account that, if \eqref{igualdad} holds true, then, for each $a\in\mathbb{R}^d$, we have that
\[
f(a)=A_a(\mathbf{0})=A_{\mathbf{0}}(-a).
\]
 In other words, $f(x)=A_{\mathbf{0}}(-x)\in \Pi_{dm,\max}^{d}$.

 Let us demonstrate the validity of \eqref{igualdad}. We fix $a\in\mathbb{R}^d$ and we define the polynomial $C(x)=A_{\mathbf{0}}(x-a)$. Let us show that $C=A_a$. Obviously, $C\in \Pi_{dm,\max}^d$. Assume that $f$ is continuous at every point of a set  $W\subseteq\mathbb{R}^d$ which is a correct interpolation set for  $\Pi_{dm,\max}^{d}$. Obviously, $C\in \Pi_{dm,\max}^d$, so that $C=A_a$ if and only if $C_{|W}=(A_a)_{|W}$. Take $\alpha\in W$. The density of $H$ in $\mathbb{R}^d$ implies that, for certain sequences of vectors $(i_{1,n},\cdots,i_{d+1,n}), (j_{1,n},\cdots,j_{d+1,n})\in\mathbb{Z}^{d+1}$, we will have that
 \begin{eqnarray*}
 \alpha&=& \lim_{n\to\infty} [(i_{1,n},\cdots,i_{d,n})+i_{d+1,n}(\theta_1,\cdots,\theta_d)]\\
 &=&  \lim_{n\to\infty} [a+(j_{1,n},\cdots,j_{d,n})+j_{d+1,n}(\theta_1,\cdots,\theta_d)].
 \end{eqnarray*}
 Hence, the continuity of $f$ at $\alpha$ implies that
  \begin{eqnarray*}
f( \alpha)&=& \lim_{n\to\infty} [f((i_{1,n},\cdots,i_{d,n})+i_{d+1,n}(\theta_1,\cdots,\theta_d))] \\
&=&  \lim_{n\to\infty} [A_{\mathbf{0}}((i_{1,n},\cdots,i_{d,n})+i_{d+1,n}(\theta_1,\cdots,\theta_d))] \\
&=& A_{\mathbf{0}}(\alpha)
\end{eqnarray*}
and
\begin{eqnarray*}
f(\alpha) &=&  \lim_{n\to\infty} [f(a+(j_{1,n},\cdots,j_{d,n})+j_{d+1,n}(\theta_1,\cdots,\theta_d))] \\
&=& \lim_{n\to\infty} [A_a(j_{1,n},\cdots,j_{d,n})+j_{d+1,n}(\theta_1,\cdots,\theta_d))] \\
&=& A_a(\alpha-a) =C(\alpha)
 \end{eqnarray*}
 It follows that $C(\alpha)=A_{\mathbf{0}}(\alpha)$ for all $\alpha\in W$. This concludes the proof that $f\in  \Pi_{dm,\max}^{d}$.

 \noindent \textbf{Part II:} Let us prove $f\in   \Pi_{m,\max}^{d}$. By Part I we know that $f\in  \Pi_{dm,\max}^{d}$. Thus, if we apply the set of equations
 \[
\Delta_{e_k}^{m+1}f(x)= 0, \text{ for } k=1,\cdots, d.
\]
to $f\in  \Pi_{dm,\max}^{d}$, we get that $f\in   \Pi_{m,\max}^{d}$. Indeed, if $f\in\Pi_{N,\max}^d$, then for every $k\in\{1,\cdots,d\}$ we can uniquely decompose $f$ as a sum
$f=\sum_{i=1}^N\phi_i(x_1,\cdots,x_{k-1},x_{k+1},\cdots,x_d)x_k^i$, with $\phi_1,\cdots,\phi_N$ polynomials in $d-1$ variables and $\phi_N\neq 0$. We must show that $N\leq m$. Now,
\[
\Delta_{e_k}^{m+1}f(x)= \sum_{i=0}^N \phi_i(x_1,\cdots,x_{k-1},x_{k+1},\cdots,x_d)\Delta_{1}^{m+1}x_k^i = \sum_{i=m+1}^N \phi_i(x_1,\cdots,x_{k-1},x_{k+1},\cdots,x_d)\Delta_{1}^{m+1}x_k^i,
\]
which is equal to zero if and only if $N\leq m$. This ends the proof for both parts.
  \end{proof}

  \begin{corollary}
 Let $H=h_1\mathbb{Z}+\cdots+h_s\mathbb{Z}$ be a finitely generated  subgroup of $\mathbb{R}^d$ and assume that
$\mathbb{Z}^d+(\theta_1,\theta_2,\cdots,\theta_d)\mathbb{Z} \subseteq H$ for certain system of real numbers $\{\theta_k\}_{k=1}^d$ such that  $\{1,\theta_1,\cdots,\theta_d\}$  is $\mathbb{Q}$-linearly independent.
 If  $f:\mathbb{R}^d \to\mathbb{R}$ satisfies
\[
\Delta_{h_k}^{m+1}f(x)= 0, \text{ for } k=1,\cdots, s
\]
and is continuous at every point of a set  $W\subseteq\mathbb{R}^d$ which is a correct interpolation set for  $\Pi_{dsm,\max}^{d}$,
then $f\in  \Pi_{sm,\max}^{d}$.
  \end{corollary}
\begin{proof}
Applying Corollary \ref{nuevo_montel} to $f$ we conclude that $\Delta_{h}^{sm+1}f=0$ for all $h\in H$. In particular, we can use Theorem \ref{M_T} with this function just substituting $m$ by $sm$.
\end{proof}

 \section{Optimality}

In this section we prove that Popoviciu's original theorem is optimal, and we consider the optimality of Theorem \ref{M_T}  in the several variables setting. Let us start with the case $d=1$. Consider the function
\[
f(x)=\left\{\begin{array}{cccc}
x(x-1)\cdots(x-(m-1)) & & x\in h_1\mathbb{Z}+h_2\mathbb{Z}\\
0 & & \text{otherwise}
\end{array}\right.,
\]
where $H=h_1\mathbb{Z}+h_2\mathbb{Z}$ is assumed to be dense in $\mathbb{R}$. If $x\in H$, then $\{x+kh_i\}_{k=0}^{m+1}\subseteq H$ for $i=1,2$, so that
$f_{|\{x+kh_i\}_{k=0}^{m+1}}=p_{|\{x+kh_i\}_{k=0}^{m+1}}$, where $p(x)=x(x-1)\cdots(x-(m-1))$. Thus, for $i\in\{1,2\}$,
\begin{eqnarray*}
\Delta_{h_i}^{m+1}f(x) &=& \sum_{k=0}^{m+1}\binom{m+1}{k}(-1)^{m+1-k}f(x+kh_i)\\
&=& \sum_{k=0}^{m+1}\binom{m+1}{k}(-1)^{m+1-k}p(x+kh_i)\\
&=& \Delta_{h_i}^{m+1}p(x)=0,
\end{eqnarray*}
since $p\in\Pi_m$. On the other hand, if $x\not\in H$, then $ \{x+kh_i\}_{k=0}^{m+1}\cap H=\emptyset$, so that $f_{|\{x+kh_i\}_{k=0}^{m+1}}=0$, and $\Delta_{h_i}^{m+1}f(x) =0$. This proves that $\Delta_{h_i}^{m+1}f(x) =0$ for all $x$, for $i=1,2$. Furthermore, $f$ is continuous at $m$ points and it is not an ordinary polynomial.

Let us now consider the case $d>1$.

Let $H=\mathbb{Z}^d+(\theta_1,\cdots,\theta_d)\mathbb{Z}$ be a dense subgroup of $\mathbb{R}^d$, and let us consider the function \\$F(x_1,\cdots,x_d)=F_1(x_1,\cdots,x_d)+\cdots+F_d(x_1,\cdots,x_d)$, where $F_i(x_1,\cdots,x_d)=g(x_i)$ and
\[
g(x)=\left\{\begin{array}{cccc}
x(x-1)\cdots(x-(m-1)) & & x\in \theta_1\mathbb{Z}+\theta_2\mathbb{Z}+\cdots +\theta_d\mathbb{Z}\\
0 & & \text{otherwise}
\end{array}\right..
\]

Let us compute $\Delta_{e_k}^{m+1}F$ and $\Delta_{(\theta_1,\cdots,\theta_d)}^{m+1}F$. First of all, it is easy to check that $\Delta_{e_k}^{m+1}F_i=0$ for $1\leq i,k\leq d$ and, hence, $\Delta_{e_k}^{m+1}F=\sum_{i=1}^d \Delta_{e_k}^{m+1}F_i=0$, for $k=1,\cdots,d$. On the other hand, if $x=(x_1,\cdots,x_d)\in\mathbb{R}^d$, then
\begin{eqnarray*}
\Delta_{(\theta_1,\cdots,\theta_d)}^{m+1}F(x) &=& \sum_{k=0}^{m+1}\binom{m+1}{k}(-1)^{m+1-k}F(x+k(\theta_1,\cdots,\theta_d))\\
&=& \sum_{k=0}^{m+1}\binom{m+1}{k}(-1)^{m+1-k}F(x_1+k\theta_1,\cdots,x_d+k\theta_d)\\
&=& \sum_{k=0}^{m+1}\binom{m+1}{k}(-1)^{m+1-k}\sum_{i=1}^dg(x_i+k\theta_i) \\
&=&\sum_{i=1}^d\left(\sum_{k=0}^{m+1}\binom{m+1}{k}(-1)^{m+1-k}g(x_i+k\theta_i) \right)\\
&=& \sum_{i=1}^d\Delta_{\theta_i}^{m+1}g(x_i).
\end{eqnarray*}
Now, if $t\in \Gamma=\theta_1\mathbb{Z}+\cdots+\theta_d\mathbb{Z}$, then $\{t+k\theta_i\}_{k=0}^{m+1}\subseteq \Gamma$, so that
$g_{|\{t+k\theta_i\}_{k=0}^{m+1}}=p_{|\{t+k\theta_i\}_{k=0}^{m+1}}$, where $p(t)=t(t-1)\cdots(t-(m-1))$, and $\Delta_{\theta_i}^{m+1}g(t)=0$. If $t\not\in \Gamma$,
then $\{t+k\theta_i\}_{k=0}^{m+1}\cap \Gamma=\emptyset$ and $\Delta_{\theta_i}^{m+1}g(t)=0$. This means that, for all $x=(x_1,\cdots,x_d)\in\mathbb{R}^d$,  $\Delta_{\theta_i}^{m+1}g(x_i)=0$. Hence $\Delta_{(\theta_1,\cdots,\theta_d)}^{m+1}F=0$.

Now $F$ is continuous at $W=\{(i_1,\cdots,i_d): 0\leq i_k\leq m-1,\text{ for all } 1\leq k\leq d\}$, which is a correct interpolation set for $\Pi_{m-1,\max}^d$. If we take into account that the conclusion of Theorem \ref{M_T} is that $f\in\Pi_{m,\max}^d$, it seems natural to claim that this example shows that  Theorem \ref{M_T} is optimal (or near optimal) also for $d>1$.


 \bibliographystyle{amsplain}

\bigskip

\footnotesize{
\noindent A. G. Aksoy,\\
Department of Mathematics. Claremont McKenna College,\\
Claremont, CA, 91711, USA.\\
{\tt aaksoy@cmc.edu}\\ \\
J. M. Almira,\\
Departamento de Matem\'{a}ticas. Universidad de Ja\'{e}n,\\
E.P.S. Linares, C/ Alfonso X el Sabio, 28,\\
23700 Linares (Ja\'{e}n) Spain.\\
{\tt jmalmira@ujaen.es}
}

\end{document}